\documentclass[11pt]{amsart}
\usepackage{amsfonts, amsthm, epsfig, amsmath, amssymb, color}
\usepackage{textcomp}
\usepackage{paralist}

\usepackage[english]{babel}

\usepackage{color}

\def\nn{\nonumber}

\def\lp{L\'{e}vy process}

\newcommand{\mcl}{\mathcal{L}}
\newcommand{\al}{\alpha}
\newcommand{\gam}{\gamma}

\newcommand{\ve}{\varepsilon}
\newcommand{\eps}{\varepsilon}

\renewcommand{\P}{\mathsf P}
\newcommand{\E}{\mathsf E}

\newcommand{\R}{\mathbb R}
\renewcommand{\Re}{\mathbb R}
\newcommand{\df}{ d}
\newcommand{\prt}{\partial}
\newcommand{\mc}[1]{\mathcal {#1}}
\newcommand{\dif}{{\mathrm D}}

\newcommand{\pr}{\mathsf P}

\newcommand{\be}{\begin{equation}}
\newcommand{\ee}{\end{equation}}
\newcommand{\ba}{\begin{aligned}}
\newcommand{\ea}{\end{aligned}}

\def\lp{L\'{e}vy process}
\def\lm{L\'{e}vy measure}

\numberwithin{equation}{section}

\theoremstyle{plain}
\newtheorem{thm}{Theorem}[section]
\newtheorem{prop}{Proposition}[section]

\newtheorem{lem}{Lemma}[section]
\theoremstyle{definition}

\newtheorem{ex}{Example}[section]
\theoremstyle{remark}
\newtheorem{rem}{Remark}[section]

\begin{document}

\title[Uniform LAN of locally stable L\'{e}vy process]{
Uniform LAN property of locally stable L\'{e}vy process observed at high frequency
}

\author{D. O. Ivanenko}
\address{Kyiv National Taras Shevchenko  University, Volodymyrska, 64, Kyiv, 01033,
             Ukraine} \email{ida@univ.kiev.ua}

\author{A. M. Kulik}
\address{Institute of Mathematics, Ukrai\-ni\-an National Academy of Sciences,
01601 Tereshchenkivska, 3, Kyiv, Ukraine}
\email{kulik.alex.m@gmail.com}

\author{H. Masuda}
\address{{1) Faculty of Mathematics, Kyushu University, 744 Motooka, Nishi-ku, Fukuoka 819-0395, Japan};
{2) CREST, JST, 744 Motooka, Nishi-ku, Fukuoka 819-0395, Japan}}
\email{hiroki@math.kyushu-u.ac.jp}
\thanks{Partly supported by JSPS KAKENHI Grant Number 26400204 (HM)}

\subjclass[2000]{Primary ; Secondary}
\date{Revised \today\ (First version November 6, 2014)}

\begin{abstract}
Suppose we have a high-frequency sample from the {\lp} of the form $X_t^\theta=\beta t+\gamma Z_t+U_t$,
where $Z$ is a possibly asymmetric locally $\al$-stable {\lp}, and $U$ is a nuisance {\lp} less active than $Z$.
We prove the LAN property about the explicit parameter $\theta=(\beta,\gam)$ under very mild conditions 
without specific form of the {\lm} of $Z$, thereby generalizing the LAN result of A\"{\i}t-Sahalia and Jacod \cite{AJ07}. 
In particular, it is clarified that a non-diagonal norming may be necessary in the truly asymmetric case. 
Due to the special nature of the local $\al$-stable property, the asymptotic Fisher information matrix takes a clean-cut form.
\end{abstract}

\maketitle

\section{Introduction}

Ever since Le Cam's pioneering work \cite{LeC60},
local asymptotics of likelihood random fields has been playing a crucial role in the theory of asymptotic inference.
Specifically, the celebrated local asymptotic normality property (LAN) introduced by Le Cam
has been a longstanding prominent concept, based on which we can deduce, among others,
asymptotic optimality criteria for estimation and testing hypothesis.
Not only for the classical i.i.d. models,
there are many existing LAN results for several kinds of statistical experiments of dependent data,
including ergodic times-series models, homoscedastic models, and ergodic stochastic processes, to mention juts a few.
One can consult \cite{LeCYan00} and the references therein for a systematic account of the LAN together with many related topics.

It is a common knowledge that verification of the LAN for a stochastic processes with no closed-form likelihood
is generally a difficult matter. In case of diffusions under high-frequency,
Gobet \cite{Gobe} and \cite{Gob02} successfully derived the LAN and LAMN by means of the Malliavin calculus.
There the structures of the limit experiments turned out to be simple enough (normal or mixed normal).
One of theoretical merits of high-frequency sampling is that it enables us to take into account
a small-time approximation of the underlying model, based on which we may derive an implementable and
asymptotically efficient estimator. This has been achieved for the diffusion models,
see Kessler \cite{Kes97} and Genon-Catalot and Jacod \cite{GJ93}.
However, to say nothing of L\'{e}vy driven non-linear stochastic differential equations,
much less has been known about the explicit LAN result for {\lp es} observed at high frequency
where the transition probability is hardly available in a closed form.
We refer to \cite{Mas_LM} for several explicit case studies about LAN result
and related statistical-estimation problems concerning {\lp es} observed at high frequency.
Especially when the underlying {\lp} or the most active part of the process is symmetric $\al$-stable,
the explicit LAN result has been proved in \cite{AJ07} and \cite{Mas09}. See also \cite{AJ08} for
the precise asymptotic behavior of the Fisher-information matrix for the same model setting as in \cite{AJ07}.

We will consider the {\lp} $X^{\theta}$ described by $X_t^\theta=\beta t+\gamma Z_t+U_t$,
where $Z$ is a \emph{locally $\al$-stable} {\lp} and where
$U$ is a {\lp} which is independent of $Z$ and \emph{less active} than $Z$,
the latter being regarded as a nuisance process; we specify them below.
The objective of this paper is to derive the LAN about the explicit parameter $\theta=(\beta,\gam)$ under very mild conditions,
when $X^\theta$ is observed at high-frequency.
Our model setting is quite broad to cover many specific examples of infinite-activity pure-jump {\lp es},
and in particular generalizes the LAN result of \cite{AJ07},
for the locally $\al$-stable property only requires that the {\lm} behaves like that of the $\al$-stable distribution only near the origin,
hence is much weaker requirement than the genuine $\al$-stable case.
It turns out that the special nature of the locally $\al$-stable character leads to a clean-cut limit experiments
described in terms of the $\al$-stable density. Owing to high-frequency sampling,
the method we propose is highly non-sensitive with respect to the nuisance process $U$,
and allows us to formulate the LAN property uniformly with respect to a class of nuisance processes;
this explains the term ``uniform'' in the title of the paper.

Our proof of the LAN property is based on two principal ingredients. One of them is the classical $L_2$-regularity technique, which dates back to Le Cam. Another important ingredient is the Malliavin calculus-based integral representation  for the derivative of the log-likelihood function, which we use in order to derive the $L_p$-bounds for this derivative. This method of proof is mainly based on the ideas developed in \cite{Kul_ivanenko_2}, \cite{Kul_ivanenko} for the model where $X^\theta$ is a solution to a L\'evy-driven SDE observed with a fixed frequency, but in the high-frequency case we encounter new challenge to design the particular version of the Malliavin calculus in a way which provide  asymptotically precise $L_p$-bounds. We mention an independent recent paper \cite{CG15}, where similar tools are developed for the same purposes. Our way to obtain the asymptotically precise $L_p$-bounds and its relation to that developed in \cite{CG15} is discussed in details in Section \ref{s3} below.

It is natural to ask for extending our LAN result for stochastic differential equation driven by a locally $\al$-stable $Z$.
This extension is far-reaching and may involve the notion of
the locally asymptotically mixed normality property (LAMN) introduced by Jeganathan \cite{Jeg82},
which covers cases of random asymptotic Fisher information matrix. This is particularly relevant to
heteroscedastic processes observed at $n$ distinct time points over a fixed time domain.
In such cases it is typical that randomness of the covariance structure is not averaged out in the limit experiments.
See \cite{Doh87}, \cite{GJ94} and \cite{Gobe} for the case of diffusion processes.
The LA(M)N property of a solution to a SDE driven by a locally $\al$-stable {\lp}
under high-frequency sampling is one of currently-projected topics.
To the best of our current knowledge,
the papers \cite{CG15} and \cite{Mai14} are the only existing result in this direction.
This will involve more technicalities than the present L\'{e}vy-process setting, and will be investigated in a subsequent paper.

\medskip

This paper is organized as follows. In Section \ref{s1} we describe the model, introduce the assumptions, and formulate the main results of the paper. Section \ref{s2} contains the main part of the proof, which is based on the Le Cam's  $L_2$-regularity technique and relies on $L_p$-bounds for  the derivative of the log-likelihood function. These $L_{p}$-bounds are proved in Section \ref{s3} by means of a specially designed version of the Malliavin calculus.


\section{Main results}\label{s1}

Let $X^\theta$ be a L\'evy process of the form
\be\label{X}
X_t^\theta=\beta t+\gamma Z_t+U_t, \quad t\geq 0.
\ee
Here $Z$ and $U$ are independent L\'evy processes defined on a probability space $(\Omega,\mathcal{F},\P)$,
and $\theta=(\beta,\gamma)^\top\in \Re^2$ is an unknown parameter subject to a statistical estimation.
We assume $Z$ to be such that in its L\'evy-Khintchine representation
$$
\E e^{i\lambda Z_t}=e^{t\psi(\lambda)},
$$
the characteristic exponent $\psi$ has the form
\be\label{ch_exp}
\psi(\lambda)=\int_{\Re}\left(e^{i\lambda u}-1-i\lambda u1_{|u|\leq 1}\right)\, \mu(du).
\ee
That is, $Z$ does not contain the diffusion term, the truncation function equals  $u1_{|u|\leq 1}$, and no additional drift term is involved.
Throughout this paper, the L\'evy measure $\mu$ is assumed to satisfy the following conditions:
\begin{itemize}
  \item[\textbf{H1.}] $\mu(du)=m(u)du$, and for some $\alpha\in (0,2)$,
  $$
  m(u)\sim \left\{
             \begin{array}{ll}
                C_+|u|^{-\alpha-1}, & u\to 0+, \\
                 C_-|u|^{-\alpha-1}, & u\to 0-,
             \end{array}
           \right.\quad C_-+C_+>0.
  $$
  \item[\textbf{H2.}] $m\in C^1(\Re\setminus\{0\})$, and there exists a constant $u_0>0$ such that the function
  $$
  \tau(u):={|u m'(u)|\over m(u)}
  $$
  is bounded on the set $\{|u|\leq u_0\}$ and satisfies
  $$
  \int_{|u|>u_0}\tau^{2+\delta}(u)\mu(du)<\infty
  $$
  for some $\delta>0$.
  \end{itemize}

Recall that for an $\alpha$-stable process its L\'evy measure has the density
\be\label{a_sta}
m_{\alpha, C_{\pm}}(u):= \left\{ \begin{array}{ll}
                C_+|u|^{-\alpha-1}, & u> 0, \\
                 C_-|u|^{-\alpha-1}, & u<0.
             \end{array} \right.
\ee
Hence \textbf{H1} requires that locally near the origin the L\'evy measure for $Z$  behaves similar to that for an $\alpha$-stable process; that is why we call $Z$ \emph{locally $\alpha$-stable}.
The constant $(C_{+}-C_{-})/(C_{+}+C_{-})\in[-1,1]$ determines the signed degree of skewness, see \cite{Zolot} for details.

Note that \textbf{H2} does not require $\tau(u)$ to be bounded for ``large'' $u$; this includes into the class of admissible $Z$ a wide range of ``stable-like'' L\'evy processes with
$$
m(u)=f(u)m_{\alpha, C_{\pm}}(u),
$$
where $f(u)\to 1, |u|\to 0$.
\begin{ex}[Tempered $\alpha$-stable process]
For either $f(u)=e^{1-\sqrt{1+u^2}}$, $f(u)=e^{-u^2}$,
or $f(u)=e^{-|u|}$, conditions \textbf{H1}, \textbf{H2} hold true, although $\tau(u)$ fails to be bounded.
  \end{ex}

\begin{ex}[Smoothly damped $\alpha$-stable process]
Let $m(u)=f(u)|u|^{-\alpha-1}1_{[-u_1, u_1]}(u)$, $u_{1}>0$, where
$f$ is continuous in $\R$, $f>0$ for $u\in[-u_1, u_1]$ with $f(u)\to 1$ as $u\to 0$,
and $f$ smoothly vanishes outside the interval $[-u_1, u_1]$
in such a way that $u\mapsto |u||f'(u)|/f(u)$ is locally bounded and moreover
$u\mapsto \{|u||f'(u)|/f(u)\}^{2+\delta}m(u)$ is $du$-integrable over the set $\{|u|\ge u_{0}\}$
for some $\delta>0$ and $u_{0}>0$. Then conditions \textbf{H1}, \textbf{H2} hold true;
note that $\tau(u)\le|u||f'(u)|/f(u)+\alpha+1$.
One particular example of such $f$ is of the form
$$f(u)=Ce^{-1/(u+u_1)-1/(u_1-u)}1_{[-u_1, u_1]}(u).$$
\end{ex}

We are focused on the following setting:
\begin{itemize}
  \item the process $X^\theta$ is \emph{discretely observed}, i.e. the $n$-th sample contains its values at the first $n$ points
$\{t_{k,n}=kh_n, k=1, \dots, n\}$ of the uniform partition of the time axis with the partition interval $h_n$;
  \item $h_n\to 0$ as $n\to\infty$, i.e. the discrete observations of $X^\theta$ have \emph{high frequency}.
\end{itemize}
We note that the terminal sampling time $nh_{n}$ may or may not tend to infinity as $n\to\infty$.

In what follows an open set $\Theta\in \Re^2$ denotes the set of possible values of the unknown parameter $\theta$; we assume that $\Theta\subset\Re\times (0,\infty),$ i.e. parameter $\gamma$ takes only positive values. Denote $\pr_n^\theta$ the law of the sample
$$
\Big\{X_{t_{k,n}}^\theta, k=1, \dots, n\Big\}
$$
in $(\Re^n, \mathcal{B}(\Re^n))$, and write
\be\nn
\mathcal{E}_n=\Big\{\R^n,\,\mathcal{B}(\Re^n),\,(\pr_{n}^\theta, \theta\in
\Theta)\Big\} \ee
for a statistical model based on this sample.

Under our conditions on the process $Z$, the law $\pr_n^\theta$ is absolutely
continuous with respect to Lebesgue measure (see Section \ref{s2}),
i.e. the model $\mathcal{E}_n$ possesses the \emph{likelihood function} $$
L_n\left(\theta;x_1,\dots,x_n\right)={\pr_{n}^\theta(dx_1 \dots dx_n)\over
\df x_{1} \dots \df x_n}$$
Denote by
$$Z_n(\theta_0,\theta; x_1,\dots,x_n)=
\frac{L_n\left(\theta;x_1,\dots,x_n\right)}{L_n\left(\theta_0;x_1,\dots,x_n\right)}$$
  the \emph{likelihood ratio} of $\pr_n^\theta$ with respect to
$\pr_n^{\theta_0}$ with the convention (anything)$/0=\infty$.

Our goal is to establish the LAN property for the sequence of statistical models $\mathcal{E}_n, n\geq 1$, specified above. Recall that
the LAN property is said to hold at a point $\theta_0\in\Theta$ with the matrix rate $\{ r(n)=
r(n,\theta_0),\ n\ge 1\}$ and the covariance matrix $\Sigma(\theta_0)$,
if for every $v$ the sampled likelihood ratio
$$Z_{n}(\theta_0, \theta_0+
r(n)v)=Z_n(\theta_0,\theta_0+r(n)v;
X^{\theta_0}_{t_{1,n}},\dots,X^{\theta_0}_{t_{n,n}})
$$
possesses representation under
\be\label{rep}
Z_{n}(\theta_0, \theta_0+
r(n)v)=\exp\left\{v^\top\Delta_n(\theta_0)-\frac{1}{2}v^\top\Sigma(\theta_0)
v+\Psi_n(v,\theta_0)\right\}
\ee
with
 \be\label{LT}
\Delta_n(\theta_0)\Rightarrow \mc{N}(0, \Sigma(\theta_0)), \quad
 n\rightarrow\infty \ee
and \be\label{Psi}
\Psi_n(v,\theta_0)\stackrel{\P}{\longrightarrow}0,\quad n\rightarrow\infty \ee
along $\pr_n^{\theta_0}$.

Put
\be\label{c_t}
c_t=t\int_{t^{1/\alpha}<|u|\leq 1} u\,\mu(du),
\ee
which is identically zero if $\mu$ is symmetric
and denote by $Z^{\alpha, C_\pm}$ the $\alpha$-stable process whose characteristic exponent has the form (\ref{ch_exp}) with the L\'evy  measure (\ref{a_sta}), where $C_+, C_-$ are given by the condition \textbf{H1}. Finally, denote by $\phi_{\alpha, C_\pm}$ the distribution density of $Z^{\alpha, C_\pm}_1$ (this density exists, see \cite{Zolot} or  Proposition \ref{prop1} below).

Now we are able to formulate our main result.

\begin{thm}\label{thm1}
Let $X^{\theta}$ be given by \eqref{X} and assume that $Z$ satisfies \textbf{H1} and \textbf{H2}, that
 \be\label{neglig}
 t^{-1/\alpha}U_t\to 0, \quad t\to 0,
 \ee
 in probability, and that
 $$
n^{-1/2}h_n^{1/\alpha-1}\to 0
 $$
 (automatic if $\al\in(0,1]$ since we are supposing that $h_{n}\to 0$).
 Then  the LAN property holds true at every point $\theta_0\in \Theta$ with
\be\label{answer}
 r(n)= n^{-1/2}\left(
        \begin{array}{cc}
         h_n^{1/\alpha-1} & c_{h_n}h_n^{-1} \\
          0 & 1 \\
        \end{array}
      \right), \quad \Sigma(\theta)=
          \left(\begin{array}{cc} \Sigma_{11}(\theta)
         & 0 \\
          0 & \Sigma_{22}(\theta) \\
        \end{array}
      \right),
\ee
where
\begin{align}
\Sigma_{11}(\theta_0)&=\gamma_0^{-2}\int_{\Re}\left({\phi'_{\alpha, C_\pm}(x)\over \phi_{\alpha, C_\pm}(x)}\right)^2\phi_{\alpha, C_\pm}(x)dx,
\nn\\
\Sigma_{22}(\theta_0)&=\gamma_0^{-2}\int_{\Re}\left(1+{x\phi'_{\alpha, C_\pm}(x)\over \phi_{\alpha, C_\pm}(x)}\right)^2\phi_{\alpha, C_\pm}(x)dx.
\nonumber
\end{align}
\end{thm}

\begin{rem}
Recall the definition of the Blumenthal-Getoor activity index of a L\'evy process $Y$ with the L\'evy measure $\mu_Y$:
\begin{equation}
\al_{Y}:=\inf\bigg\{q\ge 0:~\int_{|u|\le 1}|u|^{q}\mu_{Y}(du)<\infty\bigg\}.
\nonumber
\end{equation}
Then it is sufficient for the condition \eqref{neglig} that $\al_{U}<\al$ (see, e.g., p.362 of \cite{Sat99}); note that $\al_Z=\al$.
In this paper we are assuming that the activity index $\al$ is known.
This might seem disappointing, however, as it was clarified in \cite{AJ08} and \cite{Mas09},
 if one attempts to make joint maximum-likelihood estimation of $\al$ and the scale parameter $\gam$, one may confront the degeneracy of the asymptotic Fisher information matrix.
This degeneracy is inevitable, and how to cope with it is beyond the scope of this paper.
\label{rem_BGi1}
\end{rem}

\begin{rem}
In view of the standard theory \cite{IKh} concerning asymptotically efficient estimation of a LAN model,
Theorem \ref{thm1} suggests to seek an estimator $\hat{\theta}=(\hat{\beta}_{n}, \hat{\gam}_{n})^{\top}$ such that
\begin{equation}
r(n)^{-1}(\hat{\theta}_{n}-\theta_{0})
=\left(
        \begin{array}{c}
         \sqrt{n}h_n^{1-1/\alpha}(\hat{\beta}_{n}-\beta_{0}) - h_{n}^{-1/\al}c_{h_{n}}\cdot\sqrt{n}(\hat{\gam}_{n}-\gam_{0}) \\
         \sqrt{n}(\hat{\gam}_{n}-\gam_{0}) \\
        \end{array}
      \right)
\nonumber
\end{equation}
weakly tends to the centered normal distribution with the asymptotic covariance matrix $\Sigma(\theta_{0})$.
Observe that, when $\mu$ is asymmetric, the factor
\begin{equation}
h_n^{-1/\al}c_{h_{n}}
=h_{n}^{1-1/\al}\int_{h_n^{1/\al}<|u|\le 1}u\mu(du)
\nonumber
\end{equation}
may or may not vanish, or even may diverge, implying that the asymmetry essentially and non-trivially affect
estimation of the drift parameter $\beta$.
As a matter of fact, the necessity of non-diagonal norming seems to be non-standard in the literature:
typically, it is enough to take
$r(n)=\mathrm{diag}\{(\E|\frac{\partial}{\partial\theta_{j}}\log L_{n}(\theta;X^{\theta_0}_{t_{1,n}},\dots,X^{\theta_0}_{t_{n,n}})|^{2})_{j}\}$
whenever exists; for instance, the monograph \cite{BasSco83} is devoted to the diagonal norming.
We refer to \cite{Fah88} and \cite{Swe92} for some technical refinements of asymptotic inference by using a non-diagonal norming.
Nevertheless, we note that since $r(n)$ of \eqref{answer} is invertible,
we have no trouble in construction of an asymptotically confidence region of an asymptotically normally distributed estimator
converging at rate $r(n)$.
\end{rem}

\medskip

Under the condition (\ref{neglig}) the process $U$ is interpreted as a ``nuisance noise'',
in the sense that $U$ is less active than the ``principal'' part $Z$, as was mentioned in Remark \ref{rem_BGi1}.
A natural question is whether or not it is possible to extend Theorem \ref{thm1}
so as to make the LAN property valid not only for each single $U$,
but also \emph{uniformly} over some ``nuisance class'' $\mathfrak{U}$ of $U$.
Our method of proof of Theorem \ref{thm1} is strong enough to provide the following
\emph{uniform} LAN property in such an extended setting.

\begin{thm}\label{thm2}
Let $\mathfrak{U}$ be a class of L\'evy processes such that condition (\ref{neglig}) holds true uniformly over $U\in \mathfrak{U}$. If in addition $Z$ and $h_n$ satisfy conditions of Theorem \ref{thm1}, then for every $U\in \mathfrak{U}$ and $\theta_{0}\in\Theta$, the likelihood ratio for the discretely observed process (\ref{X}) admits a representation (\ref{rep}) with $r(n), \Sigma(\theta_0)$ specified in (\ref{answer}), and relations (\ref{LT}), (\ref{Psi})  hold true uniformly over $U\in \mathfrak{U}$.
\end{thm}

As it was explained in \cite{AJ07}, the uniform negligibility of $U$ would play an important role
for purposes of \emph{semiparametric statistical (adaptive) estimation} of $\theta$:
in p.358 of \cite{AJ07}, the authors introduce a class of possible nuisance noise distribution $\mcl(U_{1})$,
over which one can precisely formulate an asymptotically uniformly efficient estimation of $\theta$;
this in turn leads to the notion of asymptotically uniformly efficient estimator of $\theta$.
As a matter of fact, it would be possible to precisely state a uniform-in-$U$ version of the Haj\'{e}k-Le Cam convolution theorem,
which effectively clarifies the uniform asymptotic lower bound of an expected loss of any regular estimator with $r(n)$-rate of convergence;
among others, see Section 2.3 of \cite{BasSco83} and Section II.11 of \cite{IKh} for details.
How to construct an asymptotically efficient estimator would be several things, to be reported elsewhere.


\section{Proofs of Theorem \ref{thm1} and Theorem \ref{thm2}}\label{s2}

In this section we prove Theorem \ref{thm1} and outline the proof of Theorem \ref{thm2}. The key ingredient in these proofs would the $L_p$-bound for the derivative of the log-likelihood (Proposition \ref{prop2}), which we discuss in details and prove  separately in Section \ref{s3} below.

\subsection{Proof of Theorem \ref{thm1}: an outline and preliminaries}\label{s21}

Denote by $p_t(\theta; x, y)$ the transition probability density
for $X^\theta$, considered as a Markov process; in what follows we will prove that this density  exists.
Denote also
$$
g_t(\theta;x,y)={\nabla_\theta p_t(\theta;x,y)\over p_t(\theta;x,y)}=\nabla_\theta\log p_t(\theta; x,y),
\quad q_t(\theta;x,y)={\nabla_\theta p_t(\theta;x,y)\over 2\sqrt{p_t(\theta;x,y)}}=\nabla_\theta \sqrt{p_t(\theta;x,y)},
$$
assuming the derivatives to exist for $P_t(x, \cdot)$-a.a. $y$ for every fixed $x,t$.
Since $X^\theta$ has independent increments, we can write
$$
p_t(\theta; x, y)=p_t(\theta; y-x), \quad g_t(\theta; x, y)=g_t(\theta; y-x), \quad q_t(\theta; x, y)=q_t(\theta; y-x).
$$
Then the sampled {likelihood ratio} for the model can be written in the form
$$Z_{n}(\theta_0, \theta_0+
r(n)v)=\prod_{k=1}^n {p_{h_n}(\theta_0+r(n)v; X^{\theta_0}_{t_{k,n}}-X^{\theta_0}_{t_{k-1,n}})\over p_{h_n}(\theta_0; X^{\theta_0}_{t_{k,n}}-X^{\theta_0}_{t_{k-1,n}})}.
$$
Denote
$$
\eta_{k,n}^{\theta}=X^{\theta}_{t_{k,n}}-X^{\theta}_{t_{k-1,n}}, \quad 1\leq k\leq n,
$$
and observe that $p_{h_n}(\theta;\cdot)$ is the distribution density for $\eta_{kn}^\theta$. Hence the statistical model described above, after a re-sampling
 $$
(X_{t_{k,n}}^\theta)_{k=1}^n \mapsto (\eta_{k,n}^\theta)_{k=1}^n,
$$ actually is reduced to the one with a triangular array of independent observations. The LAN property for  triangular arrays of independent observations is well studied, e.g.  \cite{IKh}, Theorem II.3.1$'$,  Theorem II.6.1, and Remark II.6.2. In particular, in order to prove the required  LAN property at a point $\theta_0\in \Theta$ it would be enough for us to prove the following assertions.

\begin{itemize}
  \item[\textbf{A1}] For every $n$, the function
  $$
  \Theta\ni\theta\to \sqrt{p_{h_n}(\theta;\cdot)}\in L_2(\Re)
  $$
  is continuously differentiable; that is, the statistical experiment is \emph{regular.}
  \item[\textbf{A2}]$$\lim_{n\rightarrow\infty}\E \left|
\sum_{k=1}^n\left(r(n)^\top g_{h_n}\left(\theta_0;X^{\theta_0}_{kh_n}-X^{\theta_0}_{(k-1)h_{n}}\right)\right)^{\otimes2}
-\Sigma(\theta_0)\right|=0.$$
  \item[\textbf{A3}] For some  $\eps>0$,
$$ \lim_{n\rightarrow\infty} n\int_{\Re}\left|r(n)^\top g_{h_n}\left(\theta_0;y\right)\right|^{2+\eps}p_{h_n}\left(\theta_0;y\right)\, dy=0.
$$
  \item[\textbf{A4}]  For every $N>0$, $$
\lim_{n\rightarrow\infty}\sup_{|v|<N}
n\int_{\R}\left|r(n)^\top \left(q_{h_n}\left(\theta_0+
r(n)v;y\right)-
q_{h_n}\left(\theta_0;y\right)\right)\right|^2
dy=0.
$$
\end{itemize}

Before proving \textbf{A1}--\textbf{A4}, let us introduce some notation, formulate auxiliary statements, and make preliminary calculation.

Denote   $\tilde r(n)=n^{1/2}r(n)$. Denote also by $\tilde \Theta$ arbitrary (but fixed)  subset  of $\Theta$ such that
\be\nn
\inf\{\gamma: (\beta, \gamma)\in \tilde \Theta\}>0.
\ee
Consider the random variables
$$
\zeta_{\alpha,t}=t^{-1/\alpha}(Z_t+c_t).
$$
The following statement is proved in Appendix.
\begin{prop}\label{prop1}\begin{enumerate}
\item $\zeta_{\alpha,t}\Rightarrow Z^{\alpha,C_\pm}_1, t\to 0+$.
\item The variables $\zeta_{\alpha,t}$, $t>0$, an $Z^{\alpha,C_\pm}_1$ possess the distribution densities
$\phi_{\alpha, t}$, $t>0$, and $\phi_{\alpha, C_\pm}$, respectively.
These densities are infinitely differentiable, bounded together with their derivatives,
and for every $N>0$
\be\label{conv}
\sup_{|x|\leq N}|\phi_{\alpha,t}(x)-\phi_{\alpha, C_\pm}(x)|\to 0, \quad
\sup_{|x|\leq N}|\phi_{\alpha,t}'(x)-\phi_{\alpha, C_\pm}'(x)|\to 0, \quad t\to 0+.
\ee
\end{enumerate}
\end{prop}

Denote
$$
Y_t=X_t-U_t=\gamma t^{1/\alpha}\zeta_{\alpha,t} +\beta t-\gamma c_t,
$$
then the distribution density for $Y_t$ under $\pr^{\theta}$ equals
\be\label{ptY}
p_t^Y(\theta; x)=\gamma^{-1}t^{-1/\alpha}\phi_{\alpha,t}\Big(\gamma^{-1}t^{-1/\alpha}(x-\beta t+\gamma c_t)\Big),
\ee
and consequently
\be\nn
\ba
\prt_\beta p_t^Y(\theta; x)&=-\gamma^{-2}t^{1-2/\alpha}\phi_{\alpha,t}'\Big(\gamma^{-1}t^{-1/\alpha}(x-\beta t+\gamma c_t)\Big),\\
\prt_\gamma p_t^Y(\theta; x)&=-\gamma^{-2}t^{-1/\alpha}\bigg[ \phi_{\alpha,t}\Big(\gamma^{-1}t^{-1/\alpha}(x-\beta t+\gamma c_t)\Big)
\\&\hspace*{2cm}+
{x-\beta t\over \gamma t^{1/\alpha}}\phi_{\alpha,t}'\Big(\gamma^{-1}t^{-1/\alpha}(x-\beta t+\gamma c_t)\Big)\bigg]
\\&=-\gamma^{-2}t^{-1/\alpha}\Big[\phi_{\alpha,t}(z)+z\phi_{\alpha,t}'(z)-c_t t^{-1/\alpha}\phi_{\alpha,t}'(z)\Big]_{z=\gamma^{-1}t^{-1/\alpha}(x-\beta t+ \gam c_t)}.
\ea\ee
Because $X_t=Y_t+U_t$ and $Y,U$ are independent, we have
\be\label{pt_conv_Y}
p_t(\theta;x)=\int_{\Re}p_t^Y(\theta;x-t^{1/\alpha}y)\nu_t(dy),
\ee
where $\nu_t$ denotes the law of $t^{-1/\alpha}U_t$.

Taking (\ref{ptY}) into account, we can re-arrange the above convolution formula for $p_t(\theta;x)$ in the following way:
\be\label{pt}
p_t(\theta;x)=\gamma^{-1}t^{-1/\alpha} f_t\Big(\theta;\gamma^{-1}t^{-1/\alpha}(x-\beta t+\gamma c_t)\Big),
\ee
where
$$
f_t(\theta;z)=\int_{\Re}\phi_{\alpha,t}(z-\gamma ^{-1}y)\nu_t(dy).
$$
Note that (\ref{neglig}) implies  $\nu_t\Rightarrow\delta_0, t\to 0$, and $\gamma$ is separated from $0$ when $\theta=(\beta,\gamma)\in \tilde \Theta$. Then by the first assertion in (\ref{conv}), for every $N>0$
\be\label{pt_conv}
\sup_{\theta\in \tilde \Theta}\sup_{|z|\leq N}|f_t(\theta;z)-\phi_{\alpha, C_\pm}(z)|\to 0, \quad t\to 0+.
\ee

It follows from (\ref{pt_conv_Y}) that
$$
\prt_\beta p_t(\theta;x)=\int_{\Re}\prt_\beta p_t^Y(\theta;x-t^{1/\alpha}y)\nu_t(dy), \quad \prt_\gamma p_t(\theta;x)=\int_{\Re}\prt_\gamma p_t^Y(\theta;x-t^{1/\alpha}y)\nu_t(dy)
$$
Then, similarly as above,  we have
\be\label{gt}\ba
\prt_\beta p_t(\theta; x)&=-\gamma^{-2}t^{1-2/\alpha}f_t^{(1)}\Big(\theta;\gamma^{-1}t^{-1/\alpha}(x-\beta t+\gamma c_t)\Big),\\
\prt_\gamma p_t(\theta; x)&=-\gamma^{-2}t^{-1/\alpha}\Big[ f_t(\theta;z)+f_t^{(2)}(\theta;z)-c_tt^{-1/\alpha}f_t^{(1)}(\theta;z)\Big]_{z=\gamma^{-1}t^{-1/\alpha}(x-\beta t+\gamma c_t)}
\ea\ee
with
$$
f_t^{(1)}(\theta;z)=\int_{\Re}\phi_{\alpha,t}'(z-\gamma ^{-1}y)\nu_t(dy), \quad f_t^{(2)}(\theta;z)=\int_{\Re}(z-\gamma ^{-1}y)\phi_{\alpha,t}'(z-\gamma ^{-1}y)\nu_t(dy),
$$
and for every $N>0$
\be\label{gt_conv}
\sup_{\theta\in \tilde \Theta} \sup_{|z|\leq N}|f_t^{(1)}(\theta;z)-\phi_{\alpha, C_\pm}'(z)|\to 0,\quad \sup_{\theta\in \tilde \Theta}\sup_{|z|\leq N}|f_t^{(2)}(\theta;z)-z\phi_{\alpha, C_\pm}'(z)|\to 0, \quad t\to 0+.
\ee

Below we use formulae (\ref{pt}) -- (\ref{gt_conv}) to  control the ``local'' behavior of the functions $g_t, q_t$ involved into \textbf{A1}--\textbf{A4}. To control the ``global'' behavior, we  use the following  moment bound; we evaluate this bound in
 Section \ref{s3} below.

\begin{prop}\label{prop2} Under conditions of Theorem \ref{thm1}, for every $\tilde \Theta$  and every  $\delta_1\in (0, \delta)$ (where $\delta$ comes from \textbf{H2}),
$$
\sup_{n\geq 1, \theta\in \tilde \Theta} \E\left|\tilde r(n)^{\top}g_{h_n}\left(\theta;X^{\theta}_{h_n}\right)\right|^{2+\delta_1}<\infty.
$$
\end{prop}

\subsection{Proofs of \textbf{A1}--\textbf{A4}.}\label{s22}

\subsubsection{Proof of \textbf{A1}}
For this it is sufficient to show that, for a fixed $t=h_n$, the mapping
\be\label{cont}
\Theta\ni \theta\mapsto q_t(\theta; \cdot)\in L_2(\Re)
\ee
is continuous, and for any $\theta_1, \theta_2$ such that the segment $[\theta_1, \theta_2]$ is contained in $\Theta$,
\be\label{NL}
\sqrt{p_t(\theta_1; \cdot)}-\sqrt{ p_t(\theta_2; \cdot)}
=\bigg(\int_{0}^1 q_t((1-s)\theta_1+s\theta_2; \cdot)\,ds\bigg)^{\top}(\theta_{1}-\theta_{2})
\ee
with the integral understood in the sense of convergence of the Riemann sums in $L_2(\Re)$. The argument here is similar and simpler to the one used in the proof of Theorem 2 in \cite{Kul_ivanenko}, hence we just outline the main steps. Define
$$
\Psi_\ve(z)=\begin{cases}0,&z<\ve/2,\\ {(z-\ve/2)^2\over 2\ve^{3/2}},&z\in [\ve/2,\ve],\\ \sqrt{z}-{7\sqrt{\ve}\over 8},&z\geq \ve.\end{cases}
$$
Then, by  the construction, for $z>0$
$$
\Psi_\ve(z)\to \Psi_0(z):=\sqrt{z}, \quad \Psi_\ve'(z)\to \Psi_0'(z)={1\over 2\sqrt{z}}, \quad \ve\to 0.
$$
Because $\Psi_\ve\in C^1, \eps>0$, we have  by Proposition \ref{prop1} and (\ref{pt}), (\ref{gt}) that $\Psi_\ve (p_{t}(\theta; x))$ depends smoothly on $\theta,x$ and $$
q_{t,\eps}(\theta; x):=\nabla_\theta\Big(\Psi_\ve (p_{t}(\theta; x))\Big)=\Psi_\ve' (p_{t}(\theta; x))\nabla_\theta p_{t}(\theta; x).
$$
Then
\be\nn
\Psi_\ve\Big(p_t(\theta_1; \cdot)\Big)-\Psi_\ve\Big(p_t(\theta_2; \cdot)\Big)
=\bigg(\int_{0}^1 q_{t,\eps}((1-s)\theta_1+s\theta_2; \cdot)\,ds\bigg)^{\top}(\theta_{1}-\theta_{2}),
\ee
and to prove (\ref{NL}) it is sufficient to prove that
$$
\Psi_\ve\Big(p_t(\theta; \cdot)\Big)\to \Psi_0\Big(p_t(\theta; \cdot)\Big)= \sqrt{p_t(\theta; \cdot)}, \quad \eps\to 0
$$
in $L_2(\Re)$ for every $\theta\in \Theta$, and that
$$
q_{t,\eps}(\theta; \cdot)\to q_t(\theta; \cdot), \quad \eps\to 0
$$
in $L_2(\Re)$ uniformly in $\theta\in [\theta_1, \theta_2]$. The latter would also provide  that the function (\ref{cont}) is continuous as a uniform limit of continuous functions. Let us prove the second convergence, since the proof of the first one is similar and simpler. By the construction, we have
$0\leq \Psi_\eps'(z)\leq \Psi'_0(z)=(2\sqrt z)^{-1}$, therefore
$$
q_{t,\eps}(\theta; x)=\Psi_\ve' (p_{t}(\theta; x))\nabla_\theta p_{t}(\theta; x)=\Upsilon_\ve (p_{t}(\theta; x)) g_{t}(\theta; x),
$$
where
$$
\Upsilon_\ve(z)=\Psi_\eps'(z)z\quad
\left\{
  \begin{array}{ll} \leq (1/ 2)\sqrt{z}
    , & z>0; \\
    = (1/ 2)\sqrt{z}, & z\geq \eps.
  \end{array}
\right.
$$
Hence
$$
\int_\Re(q_{t,\eps}(\theta; x)-q_{t}(\theta; x))^2\, dx\leq {1\over 4}\int_{\{x: p_{t}(\theta; x)\leq \eps\}}\Big(g_{t}(\theta; x)\Big)^2p_{t}(\theta; x)\, dx.
$$
Recall that $t=h_n$. Take $\tilde \Theta$ in Proposition \ref{prop2} equal to the segment $[\theta_1, \theta_2]$, then
$$
\sup_{\theta\in \tilde \Theta}\int_{\Re}\Big|g_{t}(\theta; x)\Big|^{2+\delta_1}p_{t}(\theta; x)\, dx=\sup_{\theta\in \tilde \Theta}\E\Big|g_{t}(\theta; X_{t})\Big|^{2+\delta_1}<\infty.
$$
Hence by the H\"older inequality
$$
\int_\Re(q_{t,\eps}(\theta; x)-q_{t}(\theta; x))^2\, dx\leq C \left(\int_{\{x: p_{t}(\theta; x)\leq \eps\}}p_{t}(\theta; x)\, dx\right)^{\delta_1/(2+\delta_1)}.
$$
The density $p_t(\theta; x)$ is given explicitly by (\ref{pt}). Using this representation and changing the variables $z=\gamma^{-1}t^{-1/\alpha}(x-\beta t+\gamma c_t)$,  we get $$
\int_{\{x:\, p_{t}(\theta; x)\leq \eps\}}p_{t}(\theta; x)\, dx= \int_{\{z:\, f(\theta;z)\leq \gamma t^{1/\alpha}\eps\}}f(\theta;z)\, dz.
$$
We have $\phi_{\alpha,t}\in L_1(\Re)$, and therefore the mapping
$$
[\theta_1, \theta_2]\ni \theta=(\beta, \gamma)\mapsto \phi_{\alpha,t}(\cdot-\gamma^{-1}y)\in L_1(\Re)
$$
is continuous. Hence the mapping
$$
[\theta_1, \theta_2]\ni \theta\mapsto f(\theta;\cdot)=\int_{\Re}\phi_{\alpha,t}(\cdot-\gamma^{-1}y)\nu_t(dy)
$$
is continuous, as well. This finally implies that
$$
\int_{\{x:\, p_{t}(\theta; x)\leq \eps\}}p_{t}(\theta; x)\, dx= \int_{\{z:\, f(\theta;z)\leq \gamma t^{1/\alpha}\eps\}}f(\theta;z)\, dz\to 0, \quad \eps\to 0
$$
uniformly in $\theta\in [\theta_1, \theta_2]$, which completes the proof of the required convergence and provides \textbf{A1}.

\subsubsection{Proof of \textbf{A2}}
Denote $$
\Gamma_{k,n}^\theta =\tilde r(n)^\top g_{h_n}\left(\theta;X^{\theta}_{kh_n}-X^{\theta}_{(k-1)h_{n}}\right), \quad k=1, \dots, n,
$$ then
$$
\sum_{k=1}^n\left(r(n)^\top g_{h_n}\left(\theta;X^{\theta}_{kh_n}-X^{\theta}_{(k-1)h_{n}}\right)\right)^{\otimes2}={1\over n}
\sum_{k=1}^n\left(\Gamma_{k,n}^\theta\right)^{\otimes2}.
$$
Since $X$ is a L\'evy process, $\{\Gamma_{k,n}^\theta\}_{1\leq k\leq n}$ is a triangular array of random vectors, which  are row-wise independent and identically distributed. Let us analyze the common law of $\Gamma_{k,n}^\theta$ at an $n$-th row.

Denote $$
\xi^\theta_{k,n}=\gamma^{-1}h_{n}^{-1/\alpha}(X^{\theta}_{kh_n}-X^{\theta}_{(k-1)h_{n}}-\beta h_n+\gamma c_{h_n}), \quad  k=1, \dots, n,
$$
which are i.i.d. random variables with
$$
\xi^\theta_{1,n}\stackrel{d}{=}\zeta_{\alpha, h_n}+\gamma^{-1}h_{n}^{-1/\alpha}U_{h_n}.
$$
By  statement (1) of Proposition \ref{prop1} and (\ref{neglig}), we have then
\be\label{weak}
\xi^\theta_{1,n}\Rightarrow Z_1^{\alpha, C_\pm}.
\ee

Next, by (\ref{pt}), (\ref{gt}) the components of $g_{h_n}(\theta;X^{\theta}_{kh_n}-X^{\theta}_{(k-1)h_{n}})$ are given by
$$
g_{h_n}^1\left(\theta;X^{\theta}_{kh_n}-X^{\theta}_{(k-1)h_{n}}\right)=-\gamma^{-1}h_n^{1-1/\alpha}{
f_{h_n}^{(1)}(\theta;\xi^\theta_{k,n})\over f_{h_n}(\theta;\xi^\theta_{k,n})},
$$
$$
g_{h_n}^2\left(\theta;X^{\theta}_{kh_n}-X^{\theta}_{(k-1)h_{n}}\right)=-\gamma^{-1}\left[1+{f_{h_n}^{(2)}(\theta;\xi^\theta_{k,n})\over f_{h_n}(\theta;\xi^\theta_{k,n})}\right]+\gamma^{-1}c_{h_n}h_n^{-1/\alpha}{
f_{h_n}^{(1)}(\theta;\xi^\theta_{k,n})\over f_{h_n}(\theta;\xi^\theta_{k,n})}.
$$
Recall that
$$
\tilde r(n)^\top=\left(
        \begin{array}{cc}
         h_n^{1/\alpha-1} & 0 \\
          c_{h_n}h_n^{-1} & 1 \\
        \end{array}
      \right),
$$
hence we can write finally
$$
\Gamma_{k,n}^\theta=\gamma^{-1}G_{\alpha, h_n}(\theta;\xi^\theta_{k,n}),
$$
where
the vector-valued functions $G_{\alpha, h_n}$
have the components
$$
G_{\alpha, h_n}^1(\theta;x)=-{
f_{h_n}^{(1)}(\theta;x)\over f_{h_n}(\theta;x)}, \quad G_{\alpha, h_n}^2(\theta;x)=-1-{f_{h_n}^{(2)}(\theta;x)\over f_{h_n}(\theta;x)}.
$$
Denote by $G_{\alpha,C_\pm}$  the vector-valued function with the components
$$
G_{\alpha,C_\pm}^1(x)=-{\phi_{\alpha,C_\pm}'(x)\over \phi_{\alpha,C_\pm}(x)}, \quad G_{\alpha,C_\pm}^2(x)=-1-{x\phi_{\alpha,C_\pm}'(x)\over \phi_{\alpha,C_\pm}(x)},
$$
and denote  for $\eps>0, N>0$
$$
K_{\eps,N}=\{x: |x|\leq N, \phi_{\alpha,C_\pm}(x)\geq \eps\},
$$
which is a compact set in $\Re$.
It follows from (\ref{pt_conv}), (\ref{gt_conv}) that, for every fixed $\theta\in \Theta, \eps>0, N>0$,
$$G_{\alpha, t}(\theta;x)\to G_{\alpha, C_\pm}(x), \quad t\to 0+
$$
uniformly with respect to $x\in K_{\eps,N}$.
By (\ref{weak}) and continuity of the limit distribution,
\begin{equation}
\limsup_{n\to\infty}\mathsf{P}(\xi^{\theta}_{1,n} \notin K_{\varepsilon,N})
\le
\mathsf{P}( Z^{\alpha, C_{\pm}}_{1} \notin K_{\varepsilon,N}).
\nonumber
\end{equation}
We also have
$$
\P(Z^{\alpha,C_\pm}_1\not \in K_{\eps,N})\to 0, \quad \eps\to 0, \quad N\to \infty.
$$
Let us summarize: the random vectors $\Gamma_{k,n}^\theta$ are represented as images of the i.i.d. random variables $\xi_{k,n}^\theta$  under the functions $G_{\alpha, h_n}(\theta;\cdot)$, and
\begin{itemize}
  \item the common law of $\xi_{k,n}^\theta$ weakly converge to the  law of $Z^{\alpha,C_\pm}_1$;
  \item on every compact set $K_{\eps, N}$, the functions $G_{\alpha, h_n}(\theta;\cdot)$ converge uniformly to the function $G_{\alpha, C_\pm}$ which is continuous on this compact;
  \item by choosing $\eps>0$ small and $N>0$ large, the probability for  $\xi_{k,n}^\theta\in K_{\eps, N}$ can be made arbitrarily close to 1.
\end{itemize}
Because the weak convergence is preserved by continuous mappings, we deduce from the above  that
the common law of $\Gamma_{k,n}^\theta, k=1, \dots, n$ weakly converge as $n\to \infty$ to the law of $
\Gamma^\theta=\gamma^{-1}G_{\alpha, C_\pm}(Z^{\alpha, C_\pm}_1)$. On the other hand Proposition \ref{prop2} yields that  the family
$
\{(\Gamma_{k,n}^\theta)^{\otimes 2}\}
$
is uniformly integrable, hence by the Law of Large Numbers for independent random variables
$$
\E\left|{1\over n}
\sum_{k=1}^n(\Gamma_{k,n}^\theta)^{\otimes 2}-\E (\Gamma^\theta)^{\otimes 2}\right|\to 0, \quad n\to \infty.
$$
Observe  that it is an easy calculation to show that the covariation matrix for $\Gamma^\theta$ equals $\Sigma(\theta)$ given by the second identity in (\ref{answer}). Taking $\theta=\theta_0$, we complete the proof of \textbf{A2}.

\subsubsection{Proof of \textbf{A3}}
Because
$$
n\int_{\Re}\left|r(n)g_{h_n}\left(\theta;y\right)\right|^{2+\delta_1}p_{h_n}\left(\theta;y\right)\, dy=n^{-\delta_1/2}\E \left|\tilde r(n)g_{h_n}\left(\theta;X^{\theta}_{h_n}\right)\right|^{2+\delta_1},
$$
assertion \textbf{A3}  follows from Proposition  \ref{prop2} immediately.

\subsubsection{Proof of \textbf{A4}}
We have
$$
q_t(\theta;x)={1\over 2}g_t(\theta;x)\sqrt{p_t(\theta;x)},
$$
hence for any $\theta_{1},\theta_{2}\in \Theta$
$$\ba
n\int_{\R}&\left|r(n)^{\top}\left(q_{h_n}\left(\theta_2;x\right)-
q_{h_n}\left(\theta_1;x\right)\right)\right|^2\,dx\\
&\hspace*{1cm}={1\over 4 }\int_{\Re}\left|\tilde r(n)^{\top}g_{h_n}\left(\theta_2
;x\right)\sqrt{p_{h_n}(\theta_2;x)\over p_{h_n}(\theta_1;x)}-\tilde r(n)^{\top}g_{h_n}\left(\theta_1
;x\right)\right|^2p_{h_n}(\theta_1;x)\, dx
\\&\hspace*{1cm}={1\over 4 }\E\left|{1\over \gamma_2}G_{\alpha, h_n}(\theta_2,\xi_{1,n}^{\theta_1})
\sqrt{f_{h_n}(\theta_2;\xi_{1,n}^{\theta_1})\over f_{h_n}(\theta_1;\xi_{1,n}^{\theta_1})}-{1\over \gamma_1}G_{\alpha, h_n}(\theta_1,\xi_{1,n}^{\theta_1})\right|^2;
\ea
$$
we keep using the notation introduced in the proof of \textbf{A2}.
Take $\theta_1=\theta_0$, and let $\theta_n=(\beta_{n},\gamma_{n})\to \theta_0$ be arbitrary sequence. It follows from (\ref{pt_conv}), (\ref{gt_conv}) that for every fixed $\eps>0, N>0$
$$
\sup_{x\in K_{\eps, N}}\left|{1\over \gamma_n}G_{\alpha, h_n}(\theta_n,x)
\sqrt{f_{h_n}(\theta_n;x)\over f_{h_n}(\theta_0;x)}-{1\over \gamma_0}G_{\alpha, h_n}(\theta_0,x)\right|\to 0, \quad n\to \infty.
$$

 Then, by the Cauchy inequality $(a+b)^2\leq 2a^2+2b^2$,
$$\ba
&\limsup_{n\to \infty}n\int_{\R}\left|r(n)^{\top}\left(q_{h_n}\left(\theta_n
;x\right)-
q_{h_n}\left(\theta_0;x\right)\right)\right|^2\,dx\\&\leq  \limsup_{n\to \infty}\left({1\over 2\gamma_n^2}\E\left|G_{\alpha, h_n}(\theta_n,\xi_{1,n}^{\theta_0})\right|^2
{f_{h_n}(\theta_n;\xi_{1,n}^{\theta_0})\over f_{h_n}(\theta_0;\xi_{1,n}^{\theta_0})}1_{\xi_{1,n}^{\theta_0}\not \in K_{\eps, N}}+
{1\over 2\gamma_0^2}\E\left|G_{\alpha, h_n}(\theta_0,\xi_{1,n}^{\theta_0})\right|^2
1_{\xi_{1,n}^{\theta_0}\not \in K_{\eps, N}}
\right)
\\&\hspace*{1cm}=  \limsup_{n\to \infty}\left({1\over 2\gamma_n^2}\E\left|G_{\alpha, h_n}(\theta_n,\xi_{1,n}^{\theta_n})\right|^2
1_{\xi_{1,n}^{\theta_n}\not \in K_{\eps, N}}+
{1\over 2\gamma_0^2}\E\left|G_{\alpha, h_n}(\theta_0,\xi_{1,n}^{\theta_0})\right|^2
1_{\xi_{1,n}^{\theta_0}\not \in K_{\eps, N}}
\right).
\ea
$$
Since
$$
\gamma^{-1}G_{\alpha, h_n}(\theta;\xi^\theta_{1,n})=\Gamma_{1,n}^\theta=\tilde r(n)^\top g_{h_n}\left(\theta;X^{\theta}_{kh_n}-X^{\theta}_{(k-1)h_{n}}\right),
$$
we deduce using Proposition \ref{prop2} and the H\"older inequality that
$$\ba
&\limsup_{n\to \infty}n\int_{\R}\left|r(n)^{\top}\left(q_{h_n}\left(\theta_n
;x\right)-
q_{h_n}\left(\theta_0;x\right)\right)\right|^2\,dx
\\&\hspace*{1cm}\leq C \limsup_{n\to \infty}
\left(\P(\xi_{1,n}^{\theta_n}\not \in K_{\eps, N})^{\delta_{1}/(2+\delta_{1})}
+\P(\xi_{1,n}^{\theta_0}\not \in K_{\eps, N})^{\delta_{1}/(2+\delta_{1})}\right)\ea
$$
with some constant $C$. We have
$$
\xi_{1,n}^{\theta_n}\Rightarrow Z_1^{\alpha, C_\pm}, \quad \xi_{1,n}^{\theta_0}\Rightarrow Z_1^{\alpha, C_\pm}, \quad n\to \infty,
$$
hence we get
$$
\limsup_{n\to \infty}n\int_{\R}\left|r(n)\left(q_{h_n}\left(\theta_n
;x\right)-
q_{h_n}\left(\theta_0;x\right)\right)\right|^2\,dx\leq 2C\P(Z_1^{\alpha, C_\pm}\not \in K_{\eps, N})^{\delta_{1}/(2+\delta_{1})}.
$$
Recall that $\eps>0, N>0$ here are arbitrary. Taking in the above inequality $\eps\to 0, N\to \infty$ we get finally
$$
\limsup_{n\to \infty}n\int_{\R}\left|r(n)\left(q_{h_n}\left(\theta_n
;x\right)-
q_{h_n}\left(\theta_0;x\right)\right)\right|^2\,dx=0,
$$
which completes the proof of  \textbf{A4}.\qed

\medskip

\subsection{Outline of the proof of Theorem \ref{thm2}} To get the required LAN property uniformly in $U\in \mathfrak{U}$, it is enough to fix a sequence $U^n$ of L\'evy processes, such that
$$
h^{-1/\alpha}_nU^n_{h_n}\to 0
$$
in probability, and repeat the above argument for a modified statistical model, where the process $X^\theta$ in the $n$-th sample is replaced by
$$X^{\theta,n}_t=\beta t+\gamma Z_t+U^n_t.$$ The moment bound in Proposition \ref{prop2} is, to a very high extent,  insensitive with respect to the process $U$; in particular, we will show in Section \ref{s3} that
\be\label{lp_bound_ext}
\sup_{n\geq 1, \theta\in \tilde \Theta} \E\left|\tilde r(n)g_{h_n}\left(\theta;X^{\theta,n}_{h_n}\right)\right|^{2+\eps}<\infty.
\ee
The law of the process $U$ is involved in the definition of the functions $f, f^{(1)}, f^{(2)}$, but it is straightforward to see that the relations (\ref{pt_conv}), (\ref{gt_conv}) in fact hold true uniformly with respect to $U\in \mathfrak{U}$. Finally, the
random variables
$$
\xi_{k,n}^{\theta,n}=\gamma^{-1}h_{n}^{-1/\alpha}(X^{\theta,n}_{kh_n}-X^{\theta,n}_{(k-1)h_{n}}-\beta h_n+\gamma c_{h_n})\stackrel{d}{=}\zeta_{\alpha, h_n}+\gamma^{-1}h_{n}^{-1/\alpha}U_{h_n}^n
$$
weakly converge to $Z^{\alpha, C_\pm}_1$. Hence repeating, with obvious notational changes, the calculations from Section \ref{s22}, we get properties \textbf{A1}  -- \textbf{A4} for the modified model, which proves the required LAN property, uniform in $U\in \mathfrak{U}$. \qed

\section{Malliavin calculus-based integral representation  for the derivative of the log-likelihood function and related $L_p$-bounds}\label{s3}

Our main aim in this section is to prove Proposition \ref{prop2}, which is the cornerstone of the proof of Theorem \ref{thm1}.  With this purpose in mind, we give an integral representation for the derivative of the log-likelihood function by means of a certain version of the Malliavin calculus. In the diffusive case, such a representation was developed by Gobet in \cite{Gobe} and \cite{Gob02}; see also \cite{Corc}.  In the L\'evy setting, the choice of a particular design of such a calculus is a non-trivial problem, and  we discuss it in details   below.

\subsection{The main statement: formulation, discussion, and an outline of the proof}\label{s31}

In what follows,  $\nu(ds, du)$ and $\tilde \nu(ds,du)=\nu(ds, du)-ds\mu(du)$ are, respectively, the Poisson point measure and the compensated Poisson measure
from the L\'evy-It\^o representation of $Z$:
$$
Z_t=\int_0^t\int_{|u|>1}u\nu(d s, d u)+ \int_0^t\int_{|u|\leq
1}u\tilde\nu(d s, d u).
$$
We denote $$
\dif_{t} X_t^\theta=\gamma\dif_{t} Z_t=\gamma \int_0^{t}\int_{\Re}u^2\nu(d s,d u),\quad  \dif^2_{t}
X^\theta_t=2\gamma\int_0^{t}\int_{\Re}u^3\,
\nu(ds,du),
$$
the genealogy of the notation will become clear later. To
 define the stochastic integrals with respect to $\nu$ properly, we decompose them in two integrals (which is a usual trick). The ``small jump'' parts, which correspond to values $u\in [-1, 1]$, are well defined  because  the functions $u^2, u^3$ are integrable with respect to $\mu$ on $[-1,1]$; respective integrals with respect to $\nu$  are understood in $L_1$ sense. The ``large jump'' parts of the integrals with respect to $\nu$ are understood in the path-wise way, i.e. as sums over finite set of jumps. Next, we denote
$$
\chi(u)=-u^2{m'(u)\over m(u)}-2u
$$
and put
$$
\delta_{t}(1)=\int_0^t\int_{|u|\leq u_0}\chi(u)\tilde
\nu(d s,d u)+\int_0^t\int_{|u|> u_0}\chi(u)
\nu(d s,d u)+ tu_0^2\Big[m(u_0)-m(-u_0)\Big],
$$
where $u_0$ comes from the condition \textbf{H2}. By \textbf{H2}, $|\chi(u)|\leq C|u|$ for $|u|\leq u_0$, hence the ``small jump'' integral above is well defined in $L_2$ sense; the ``large jump'' integral is understood in the path-wise sense.

We define the \emph{modified Malliavin weight} (we postpone for a while the explanation of the terminology) as the vector $\Xi^\theta_t=(\Xi^{\beta}_{t}, \Xi^{\gamma}_{t})^\top$ with
\be\label{modif_weight}\Xi^{\beta}_{t}={t \delta_{t}(1)\over \dif_{t } X^{\theta}_t}+{t
\dif^2_{t} X^{\theta}_t \over (\dif_{t}
X^{\theta}_t)^2},\quad \Xi^{\gamma}_{t}={Z_t \delta_{t}(1)\over
\dif_{t} X^{\theta}_t}+{Z_t \dif^2_{t} X^{\theta}_t \over (\dif_{t}
X^{\theta}_t)^2}-{1\over \gamma}.\ee

Denote by $\E^{t,\theta}_{x,y}$ the expectation with respect to the law of the \emph{bridge} of the process $X^\theta$ conditioned by $X_0^\theta=x, X_t^\theta=y$. Note that because the process $X^\theta$ possesses a continuous transition probability density $p_t(\theta;x,y)=p_t(\theta;y-x)$, the law of the {bridge} is well defined for any
$t,x,y$ such that $p_t(\theta;x,y)>0$ (cf. \cite{bridges}).

The main statement in this section is the following.

\begin{thm}\label{thm3}
\begin{enumerate}
\item Let $\delta$ be the same as in \textbf{H2}. Then for every $\delta_1\in (0, \delta)$ and every $\tilde \Theta$,
\be\label{un_xi_bound}
\sup_{\theta\in \tilde \Theta}\sup_{n\geq 1}\E\Big|\tilde r(n)^\top \Xi^{\theta}_{h_n}\Big|^{2+\delta_1}<\infty.
\ee

\item The following intergal representation formula holds true:
\be\label{modif_g}
g_t(\theta; x)=\begin{cases}\E^{t,\theta}_{0,x}\Xi^{\theta}_{t}, &  p_t(\theta; x)>0,\\
0,&\hbox{otherwise}.\end{cases} \ee

\end{enumerate}
\end{thm}

It follows from (\ref{modif_g}) that
$$
\tilde r(n)^\top g_t(\theta;X^\theta_{h_n})=\E\Big[\tilde r(n)^\top \Xi_{h_n}^\theta\Big|X^\theta_{h_n}\Big].
$$
Note that the explicit formula for the weight $\Xi^\theta_t$ does not involve the ``nuisance noise'' $U$ at all, and the dependence of $p_t(\theta;X_{h_n}^\theta)$ on $U$ is contained in the operation of the conditional expectation, only. Then by the Jensen inequality
$$
\E\Big|\tilde r(n)^\top g_t(\theta;X^\theta_{h_n})\Big|^{2+\delta_1}\leq \E\Big|\tilde r(n)^\top \Xi_{h_n}^\theta\Big|^{2+\delta_1},
$$
where $U$ is not involved in the right hand side term. Hence Theorem \ref{thm3} immediately yields both Proposition \ref{prop2} and the moment bound (\ref{lp_bound_ext}).

Let us explain the main idea which  Theorem \ref{thm3} is based on.  By analogy to  Gobet's results in the diffusive case (\cite{Gobe}, \cite{Gob02}), one can naturally expect that  an integral representation of the form (\ref{modif_g}) could be obtained by means of a proper version of the Malliavin calculus for jump processes. Two possible ways  to do that  were developed in \cite{CG15} and \cite{Kul_ivanenko}, being in fact close to each other and, heuristically, being based on ``infinitesimal perturbation of the jump configuration'' with an intensity function $\rho$ which is a ``compactly'' supported smooth function (see a more detailed exposition in Section \ref{s32} below). Using either of these two approaches it is possible to prove an analogue of (\ref{modif_g}) with $\Xi_t^\theta$ being replaced by some $\Xi_{t,\rho}^\theta$ which, in full analogy to Gobet's approach, has  the meaning of the \emph{Malliavin weight} (see formula (\ref{Xi}) below). However, we then encounter following two difficulties, both being related with moment bounds for the corresponding terms.

\begin{itemize}
  \item In order to provide that $\Xi_{t,\rho}^\gamma$  is square integrable (which is necessary for $\Xi^\theta_{t,\rho}$ to have representation
(\ref{Xi}) with the Skorokhod integral in the right hand side), we need an additional moment bound for $Z$: for some $\delta'>0$,
\be\label{add} \int_{|u|\geq
1}|u|^{2+\delta'}\mu(du)<\infty. \ee
This excludes from the consideration ``heavy tailed'' L\'evy processes, e.g. the particularly important $\alpha$-stable process.
  \item Even if we confine ourselves by the class of ``light tailed'' L\'evy  processes satisfying (\ref{add}), we can not obtain analogue of the moment bound (\ref{un_xi_bound}) for $\Xi_{t,\rho}^\theta$: namely,  respective upper bound for $L_{2+\delta_1}$-norm of $\Xi_{t,\rho}^\theta$ would explode as $t\to 0+$: see Remark \ref{r2} below for detail.
\end{itemize}

Both of these ``moment'' difficulties are resolved when we put formally
\begin{equation}
\rho(u)=u^2, \quad u\in \Re,
\nonumber
\end{equation}
in the formula for $\Xi_{t,\rho}^\theta$; see Section \ref{s32} below, and especially Remark \ref{r2} which explains the heuristics behind the particular choice $\rho(u)=u^2$. This explains both the name ``the modified Malliavin weight'' we have used for  the term defined by (\ref{modif_weight}), and the background for the notation $\dif_t, \delta_t$: we take the explicit formulae for the Malliavin derivative $\dif_{t,\rho}$ and respective Skorokhod integral $\delta_{t,\rho},$ and put therein $\rho(u)=u^2$.   Note that because $\rho(u)=u^2$ is not compactly supported,  $\dif_t X_t^\theta$ may fail to be square integrable, which means that now $\dif_t X_t^\theta$ can not be interpreted as a Malliavin derivative. As a consequence, now one can not apply the Malliavin calculus tools  to prove  (\ref{modif_g}) directly. Hence we will use the following  three step procedure to prove (\ref{modif_g}):
\begin{enumerate}
\item first, we apply Malliavin calculus tools to prove analogue of (\ref{modif_g}) for $\Xi_{t,\rho}^\theta$ with compactly supported $\rho$ under the additional moment  condition (\ref{add});
\item second, we approximate $\rho(u)=u^2$ by a sequence of compactly supported $\rho$'s;
\item finally, we approximate general $Z$ by a sequence of ``light tailed'' L\'evy processes $Z^L, L\geq 1$, each of them satisfying (\ref{add}).
\end{enumerate}

\subsection{Proof of (\ref{un_xi_bound}): Moment bound}\label{s32}

First, we give an explicit expression for $\tilde r(n)^\top \Xi_t^\theta$. Denote
$$
\tilde Z_t=Z_t+c_t=\int_0^t\int_{|u|>t^{1/\alpha}}u\nu(d s, d u)+ \int_0^t\int_{|u|\leq
t^{1/\alpha}}u\tilde\nu(d s, d u),
$$
then
\be\label{explicit}
\tilde r(n)^\top \Xi_t^\theta= \left( {t^{1/\alpha} \delta_{t}(1)\over \dif_{t } X^{\theta}_t}+{t^{1/\alpha}
\dif^2_{t} X^{\theta}_t \over (\dif_{t}
X^{\theta}_t)^2} ,
                      {\tilde Z_t \delta_{t}(1)\over
\dif_{t} X^{\theta}_t}+{\tilde Z_t \dif^2_{t} X^{\theta}_t \over (\dif_{t}
X^{\theta}_t)^2}-{1\over \gamma}
                                                              \right)^\top.
\ee
 We will conclude the required bound (\ref{un_xi_bound}) from a sequence of auxiliary estimates
for the terms involved in the explicit expression (\ref{explicit}).

  Denote
 $$\kappa_t=\int_0^t\int_{|u|\leq t^{1/\alpha}}u^2\nu(ds,du).
 $$
\begin{lem}\label{l3} For every $p\geq 1$
there exists $C_p<\infty$ such that
$$
\E(t^{-2/\alpha}\kappa_t)^{-p}\leq C_p, \quad t\in (0,
1].
$$
\end{lem}
\begin{proof} For any $\eps<1$ we have
$$
\P(\kappa_t<\eps^{2} t^{2/\alpha})\leq P\Big(\nu([0,t]\times \{|u|\in
[\eps t^{1/\alpha}, t^{1/\alpha}]\})=0\Big)=\exp\Big\{-t\mu(|u|\in [\eps
t^{1/\alpha}, t^{1/\alpha}])\Big\}.
$$
Condition \textbf{H1} yields that, with some positive
constant $C$,
$$
t\mu(|u|\in [\eps t^{1/\alpha}, t^{1/\alpha}])\geq C t\int_{\eps
t^{1/\alpha}}^{t^{1/\alpha}} \alpha
|u|^{-\alpha-1}du=C(\eps^{-\alpha}-1), \quad t\in (0, 1].
$$
Hence for the family of random variables
$t^{-2/\alpha}\kappa_t, t\in (0,1]$ we have the uniform bound
$$
\P(t^{-2/\alpha}\kappa_t<\eps^{2})\leq e^{-C\eps^{-\alpha}+C}, \quad \eps<1, \quad
t\leq 1,
$$
which proves the required statement.
\end{proof}

Because
$$
\dif X^\theta_t=\gamma \int_0^t\int_{\Re}u^2\nu(ds,du)\geq \gamma \kappa_t,
$$
Lemma \ref{l3} immediately gives the following: for every $p\geq 1$,
\be\label{est1}
\sup_{\theta\in\tilde \Theta, t\in (0,1]}\E\left(t^{1/\alpha}\over \sqrt{\dif_t X^\theta_t}\right)^p<\infty.
\ee
Next, observe that both $\dif_t X^\theta_t$ and $\dif_t^2 X^\theta_t$  are represented as sums over the set of jumps of the process $Z$. Because
$$
\left(\sum_ia_i\right)^{3/2}\geq \sum_ia_i^{3/2}, \quad \{a_i\}\subset [0, \infty),
$$
we have
\be\label{est2}
\left|{\dif_t^2 X^\theta_t\over (\dif_t X^\theta_t)^{3/2}}\right|\leq {2\over \sqrt{\gamma}}.
\ee

\begin{lem}\label{lem3} For every $p\geq 1$,
\be\label{est4}
\sup_{\theta\in\tilde \Theta, t\in (0,1]}\E\left(\tilde Z_t\over \sqrt{\dif_t X^\theta_t}\right)^p<\infty.
\ee
\end{lem}
\begin{proof} We have
$$
\tilde Z_t=\int_0^t\int_{|u|\leq t^{1/\alpha}}u\tilde \nu
(ds,du)+\int_0^t\int_{|u|>t^{1/\alpha}}u\nu (ds,
du)=:\xi_t+\zeta_t,
$$
$$
\dif X^\theta_t=\gamma\int_0^t\int_{\Re}u^2\, \nu
(ds,du)=\gamma\left(\kappa_t+\eta_t\right), \quad \eta_t=
\int_0^t\int_{|u|>t^{1/\alpha}}u^2\, \nu (ds,du)
$$
($\kappa_t$ is already defined above).
Then
\be\label{11}
\left|{\tilde Z_t \over \sqrt{\dif X^\theta_t}}\right|\leq
\frac{1}{\gamma}\left({|\xi_t|\over
\sqrt{\kappa_t+\eta_t}}+{|\zeta_t|\over \sqrt{\kappa_t+\eta_t}}\right)\leq \frac{1}{\gamma}\left({|\xi_t|\over
\sqrt{\kappa_t}}+{|\zeta_t|\over \sqrt{\eta_t}}\right).
\ee

By Lemma \ref{l3},  the family
$$ \left\{{t^{1/\alpha}\over \sqrt{\kappa_t}}\right\}_{t\in (0,1]}
$$ has bounded $L_p$-norms for any $p\geq 1$. In addition, the family
$$ \left\{t^{-1/\alpha}
\xi_t\right\}_{t\in [0,1]}
$$
also has bounded $L_p$-norms for any $p\geq 1$. To see this, observe  that $\xi_t$ is an integral of a deterministic function over a compensated Poisson point measure, and therefore its exponential moments can be expressed explicitly:
$$
\E \exp(c \xi_t)=\exp\Big[t\int_{|u|\leq
t^{1/\alpha}}(e^{cu}-1-cu)\, \mu(du)\Big].
$$
Taking  $c=\pm t^{-1/\alpha}$ and using \textbf{H1}, we get
$$
\E \exp\Big(\pm t^{-1/\alpha}\xi_t\Big)\leq
\exp\Big[C_1t\int_{|u|\leq t^{1/\alpha}}(t^{-1/\alpha} u)^{2}
\mu(du)\Big]\leq C_2,
$$
which yields the required $L_p$-bounds. Applying the Cauchy inequality, we get finally that the family
$$ \left\{{|\xi_t|\over  \sqrt{\kappa_t}}\right\}_{t\in (0,1]}$$
has bounded $L_p$-norms.

For the second summand in the right hand side of (\ref{11}), we write the Cauchy inequality:
$$
{|\zeta_t|\over \sqrt{\eta_t}}\leq \sqrt{N_t}, \quad
N_t=\nu([0,t]\times\{|u|>t^{1/\alpha}\}).
$$
Observe that $N_t$ has a Poisson law with the intensity
$$
t\mu(|u|>t^{1/\alpha}), \quad t\in(0,1],
$$
which is bounded because of \textbf{H1.} Hence the family
$$ \left\{{|\zeta_t|\over \sqrt{\eta_t}}\right\}_{t\in [0,1]}$$
also has bounded $L_p$-norms, which  completes the proof of (\ref{est4}).\end{proof}

\begin{lem}\label{lem4}  Let $\delta$ be the same as in \textbf{H2}. Then
\be\label{est31}
\sup_{\theta\in\tilde \Theta, t\in (0,1]}\E\left(\delta_t(1)\over \sqrt{\dif_t X^\theta_t}\right)^{2+\delta}<\infty.
\ee
\end{lem}
\begin{proof}

The proof is similar to the previous one, but some additional technicalities arise
because now we have $\chi(u)$ instead of $u$ under the integrals in the numerator. We have
$$\ba
 \delta_t(1)&=\int_0^t\int_{|u|\leq t^{1/\alpha}}\chi(u)\tilde \nu
(ds,du)+\int_0^t\int_{|u|>t^{1/\alpha}}\chi(u)\nu (ds,
du)+t^{1+2/\alpha}\Big[m(t^{1/\alpha})-m(-t^{1/\alpha})\Big]
\\&=:\hat\xi_t+\hat\zeta_t+\varpi_t.
\ea
$$
Denote $t_0=(u_0)^\alpha$ with $u_{0}$ coming from \textbf{H2}, then the ratio
$$
\upsilon(u):={\chi(u)\over u}=-{um'(u)\over m(u)}-2
$$
is bounded on the set $\{|u|\leq t^{1/\alpha}\}\subset \{|u|\leq u_0\}$. Then the same argument as we have used before shows that the family
$$ \left\{{|\hat\xi_t|\over  \sqrt{\kappa_t}}\right\}_{t\in (0,t_0]}$$
has bounded $L_p$-norms for every $p\geq 1$.

Next,  we have by \textbf{H1} that
$$
\varpi_t\sim(C_+-C_-)t^{1/\alpha}, \quad t\to 0+,
$$
and thus
$$
|\varpi_t|\leq C t^{1/\alpha}, \quad t\in (0,1].
$$
Then by Lemma \ref{l3} the family
$$ \left\{{\varpi_t \over  \sqrt{\kappa_t}}\right\}_{t\in (0,t_0]}$$
has bounded $L_p$-norms for every $p\geq 1$.

Finally, for $t\leq t_0$ by the Cauchy inequality we have
$$
|\hat \zeta_t|\leq \sqrt{\eta_t}\sqrt{J_t},
$$
where
$$
\begin{aligned}
J_t&=\int_0^t\int_{|u|>t^{1/\alpha}}\upsilon^2(u)\nu(ds,du)\\&=
\int_0^t\int_{t^{1/\alpha}< t\leq u_0}\upsilon^2(u)\nu(ds,du)+\int_0^t\int_{|u|>u_0}\upsilon^2(u)\nu(ds,du)=:J_t^1+J_t^2.
\end{aligned}
$$
The function $\upsilon(u)$ is bounded on $\{|u|\leq u_0\}$, hence
$$
J_t^1\leq CN_t, \quad t\in (0, t_0],
$$
where $N_t$ is the same as in the  proof of Lemma \ref{lem3}. Hence $J_t^1, t\leq t_0$ have bounded $L_p$-norms for every $p\geq 1$. The random variable $J_t^2$ has  a compound Poisson distribution with the intensity of the Poisson random variable equal  $t\mu(|u|> u_0)$, and the law of a single jump equal to the image under $\upsilon$ of the measure $\mu$ conditioned to $\{|u|>u_0\}$. By condition \textbf{H2}, this law have a finite moment of the order $2+\delta$, therefore the variables $J_t^2, t\leq t_0$ have bounded $L_{2+\delta}$-norms. Summarizing all the above, we get
$$
\sup_{\theta\in\tilde \Theta, t\in (0,t_0]}\E\left(\delta_t(1)\over \sqrt{\dif_t X^\theta_t}\right)^{2+\delta}<\infty.
$$
The same bound for $t\in [t_0, 1]$ can be proved in a similar and simpler way; in that case instead of taking the integrals with respect to $\{|u|\leq t^{1/\alpha}\}, \{|u|> t^{1/\alpha}\}$, one should consider, both in the numerator and the denominator, the integrals with respect to $\{|u|\leq u_0\}, \{|u|> u_0\}$.
\end{proof}

Now we deduce (\ref{un_xi_bound}) by simply applying the H\"older inequality to \eqref{explicit}
with the estimates (\ref{est1})--(\ref{est4}) and (\ref{est31}).

\begin{rem}\label{r2} Now we can explain the main idea, which the choice of the intensity function $\rho(u)=u^2$ is based on.
When $\rho$ is compactly supported as was in \cite{Kul_ivanenko_2}, the ``large jumps'' are excluded from the formula for $\dif_{t, \rho} X_t^\theta$. On the other hand, ``large jumps'' are involved e.g. into $\tilde Z_t$, which will appear in the numerator in one term in (\ref{Xi}). We have  $\rho(u)=0, |u|>u_*$ for some $u_*>0$, and hence  the integrals
$$
\int_0^t\int_{|u|> u_*}u \nu(ds, du), \quad \int_0^t\int_{\Re}\rho(u) \nu(ds, du)
$$
are independent. In addition, we know that
$$
\E\left(\int_0^t\int_{|u|> u_*}u \nu(ds, du)\right)^2=t\int_{|u|> u_*}u^2\mu(du)+t^2\left(\int_{|u|> u_*}u^2\mu(du)\right)^2,
$$
$$
t^{-2/\alpha}\int_0^t\int_{\Re}\rho(u) \nu(ds, du)\Rightarrow \zeta, \quad t\to 0,
$$
where $\zeta$ is a positive $(\alpha/2)$-stable variable. Using this, it is easy to deduce a lower bound
$$
\E\left(\int_0^t\int_{|u|> u_0}u \nu(ds, du)\over \sqrt{\dif_{t, \rho} X_t^\theta}\right)^2\geq Ct^{-2/\alpha+1},
$$
which is unbounded for small $t$ because $\alpha<2$. This indicates that, in the present high-frequency sampling setting,
one can hardly expect to get the \emph{uniform} moment bound of the type (\ref{un_xi_bound}) for a Malliavin weight
which corresponds to a compactly supported $\rho$. Nevertheless,
in the modified construction we ``extend the support'' of $\rho$; this brings a ``large jumps'' part to the denominator, which provides a good
balance to respective parts which appear in the numerator, and this is the reason why the modified weight satisfies the required uniform moment bounds.
\end{rem}

\begin{rem}\label{r3} Another natural possibility to design the Malliavin weight  is to take into account the local scale for the  process $Z$ and to make the function $\rho$  depend on $t$ in the following way:
$$
\rho(u)=\rho_t(u)=u^2\varsigma(t^{-1/\alpha}u)
$$
with $\varsigma\in C^1$ such that $\varsigma(u)=1, |u|\leq 1$ and $\varsigma(u)=0, |u|\geq 2$.  Actually, the ``scaled'' choice of $\rho=\rho_t$ with the size of its  support $\asymp t^{1/\alpha}$  is essentially the one  used in the Malliavin calculus construction developed in \cite{CG15}.
 It is easy to see that, under such a choice,  an analogue of (\ref{est31}) would hold true; the reason  is that now $\delta_t$ would contain  only ``small jumps part'', which is well balanced with $\sqrt{\dif X_t^\theta}$ in completely the same way we have seen in the proof of Lemma \ref{lem3}. This would give the uniform moment bound for the \emph{first} component of the respective Malliavin weight, and hence the ``scaled'' choice of $\rho=\rho_t$ is appropriate when the unknown parameter  is involved into the drift term, only. However, the model with the parameter involved into the jump term this choice does not seem to be appropriate by the reason stated in Remark \ref{r2}: the ``large jump part'' of $\tilde Z_t$ is not well balanced by the ``small jump part'' $\sqrt{\kappa_t}$ of $\sqrt{\dif X_t^\theta}$. 
In regard to this point, for the sake of reference let us discuss \cite{CG15} in a little bit more detail. 
Therein the authors proved the drift-parameter LAMN property
when observing a sample $(X_{j/n})_{j=0}^{n}$ from the process $X$ of the form
\begin{equation}
X_{t} = x_{0} + \int_{0}^{t}b(X_{s},\beta)ds + Z_{t}.
\nonumber
\end{equation}
Their technical assumptions are: (i) the boundedness of the {\lm} of $Z$, which is assumed to be locally $\al$-stable in a neighborhood of the origin;
(ii) the smoothness and boundedness of $(x,\beta)\mapsto b(x,\beta)$; and that (iii) the stable-like index $\al\in(1,2)$,
imposed just for ignoring the presence of the drift term in looking at the ``scaled'' increment
$n^{-1/\al}(X_{j/n}-X_{(j-1)/n})$, which is to be close to $L_{j/n}-L_{(j-1)/n}$.
In particular, the boundedness (ii) seems essential in their proof, as well as the fact that only the drift parameter is the subject to statistical estimation. We emphasize that because of absence of the jump-related parameter $\gamma$, the result of \cite{CG15} is not comparable neither with our current results nor even with those from the aforementioned paper \cite{AJ07}. We expect that our proof technique based on the modified Malliavin weight combined with the approximation of the intensity function $\rho$
will be workable for the general case of state-dependent coefficients including the jump coefficient which contains a parameter $\gamma$; this is a subject of a further research.
\end{rem}

\subsection{Proof of (\ref{modif_g}): Integral representation}\label{s33}

The proof consists of three steps outlined at the end of Section \ref{s31}. The first step is based on a version  of the  Malliavin calculus on a space of trajectories of a L\'evy process, which we outline below and which is essentially developed in  \cite{Kul_ivanenko}. Note that the  Malliavin calculus for L\'evy noises is a classical and well developed tool, which dates back to \cite{Bichteler_Grav_Jacod} and \cite{Bismut_jumps}. However, we found it difficult to apply existing technique directly for our purposes: the reason is that unlike in the classical approach developed in \cite{Bichteler_Grav_Jacod} and \cite{Bismut_jumps}, we are interested not in the distribution density $p_t(\theta; x,y)$ itself (which is typically  treated by means of the inverse Fourier transform), but in the ratio $\prt_\theta p_t(\theta; x,y)/p_t(\theta; x,y)$. To get the $L_p$  bounds for this ratio, and especially to make approximating procedures outlined at the end of Section \ref{s31}, we need to have an integral representation for this ratio in a most simple possible form. For that purpose mainly, and also to make the exposition self-consistent, we introduce a  specially designed simple version of the Malliavin calculus. This version of course is neither  a substantial novelty nor is unique possible one; see e.g. the construction in \cite{CG15} aimed at similar purposes.

 Let the L\'evy measure $\mu$ of $Z$ satisfy  \textbf{H1}, \textbf{H2} and assume additionally that $Z$ is ``light tailed'' in the sense that (\ref{add}) holds true for some $\delta'>0$. Fix some function $\rho\in C^2$ such that  $\rho(u)=u^2$ in a neighborhood  of the point $u=0$. Consider a flow $Q_c, c\in \Re$ of transformations of $\Re$, which satisfies
$$
{d\over dc}Q_c(u)=\rho(Q_c(u)),\quad Q_0(u)=u,
$$
and for a fixed $t>0$ define respective family ${\mathcal Q}_c^t, c\in \Re$ of transformations of the process $Z_s, s\geq 0$ by the following convention:
the process ${\mathcal Q}_c^tZ$ has jumps at the same time instants with the initial process $Z$;  if the process has a jump with the amplitude $u$ then at the time moment $s$, respective jump of $Z$ has the amplitude equal either  $Q_c(u)$ or $u$ if $s\leq t$ or $s>t$, respectively. It is proved in \cite{Kul_ivanenko}, Proposition 1, that under conditions  \textbf{H1}, \textbf{H2} the law of ${\mathcal Q}_c^tZ$ in $\mathbb{D}(0, \infty)$ is absolutely continuous with respect to the law of $Z$. Hence every transformation ${\mathcal Q}_c^t$ can be naturally extended to a transformation of the  space $L_0(\Omega, \sigma(Z), \P)$ of the functionals of the process $Z$. Denote this transformation by the same symbol ${\mathcal Q}_c^t$, and call a random variable $\xi\in L_2(\Omega, \sigma(Z), \P)$ stochastically differentiable if there exists the mean square limit
$$
\hat \dif \xi=\lim_{\eps\to 0}{{\mathcal Q}_\eps^t\xi-\xi\over \eps}.
$$
The $L_2(\Omega, \sigma(Z), \P)$-closure  of the operator $\hat \dif$  is called the \emph{stochastic derivative} and is denoted by $\dif$.  The
adjoint operator $\delta=\dif^*$ is called the \emph{divergence
operator} or the \emph{extended stochastic integral}. The operators $\dif, \delta$ are well defined under conditions \textbf{H1}, \textbf{H2} for every $t>0$ and $\rho$ specified above; see \cite{Kul_ivanenko}, Remark 3.

Clearly, the above construction depends on the choice of $t$ and $\rho$: to track this dependence we use the notation $\dif_{t,\rho}, \delta_{t, \rho}$ instead of $\dif, \delta$. In a slightly larger generality, literally the same construction can be made on the space $L_0(\Omega, \sigma(Z, U), \P)$ of the functionals of the pair of processes $Z$ and $U$, with the trajectories of $U$ not being perturbed by ${\mathcal Q}_c^T$.
Then, analogously to the calculations made in \cite{Kul_ivanenko}, Sections 3.1, 3.2,  we have
$$
\dif_{t, \rho} Z_t=\int_0^{t}\int_{\Re}\varrho(u)\nu(d s,d u), \quad \dif_{t, \rho} U_t=0,
$$
$$\delta_{t, \rho}(1)=\int_0^t\int_{\Re}\chi_\rho(u)\tilde
\nu(d s,d u), \quad \chi_\rho(u)=-\varrho(u){m'(u)\over m(u)}-\varrho'(u);
$$
recall that $\nu$ and $\tilde \nu$ are, respectively, the Poisson point measure and the compensated Poisson measure
from the L\'evy-It\^o representation of $Z$.
Respectively,
$$
\dif_{t, \rho}X_t^\theta=\gamma \dif_{t, \rho} Z_t=\gamma \int_0^{t}\int_{\Re}\varrho(u)\nu(d s,d u).
$$
Furthermore, the second order stochastic derivative of $X^\theta_t$ is well defined:
$$\dif^2_{t, \rho}
X^\theta_t=\gamma\int_0^{t}\int_{\Re}\varrho(u)\varrho'(u)\,
\nu(ds,du).
$$

By Proposition \ref{prop1} and formula (\ref{pt}), the variable  $X_t^\theta$ has a distribution density $p_t(\theta; x)$ which is a $C^2$-function with respect to $\theta, x$. On the other hand, the Malliavin calculus developed above allows one to derive an integral representation for the ratio
$$
g_t(\theta;x,y)={\nabla_\theta p_t(\theta;x,y)\over p_t(\theta;x,y)}.
$$
Namely, repeating literally the proof of the assertion III of Theorem 1 in \cite{Kul_ivanenko}, we obtain the following representation:
\be\label{g}
g_t(\theta; x)=\begin{cases}\E^{t,\theta}_{0,x}\Xi^{\theta}_{t, \rho}, &  p_t(\theta; x)>0,\\
0,&\hbox{otherwise},\end{cases} \ee
where
\be\label{Xi}
\Xi^{\theta}_{t, \rho}:=\delta_{t, \rho}\left({\nabla_\theta X^{\theta}_t\over
\dif_{t, \rho} X^{\theta}_t}\right)={(\delta_{t, \rho}(1))(\nabla_\theta X^{\theta}_t)
\over \dif_{t, \rho} X^{\theta}_t}+{
(\dif^2_{t, \rho} X^{\theta}_t)(\nabla_\theta X^{\theta}_t)\over (\dif_{t, \rho}
X^{\theta}_t)^2}-{\dif_{t, \rho}(\nabla_\theta X^{\theta}_t)\over \dif_{t, \rho}
X^{\theta}_t}. \ee
Note that, formally, we can not apply Theorem 1 of \cite{Kul_ivanenko}  directly, because now we have an additional process $U$ which our target process $X^\theta$ depends on. Nevertheless, because  $U$ is not perturbed under the transformations ${\mathcal Q}_c^t$ which give the rise for the Malliavin calculus construction, it is easy to check that literally the same argument as the one used in the proof of Theorem 1  \cite{Kul_ivanenko} can be applied in the current (slightly extended) setting.

Recall that we already know $p_t(\theta;x,y)$ exists and is smooth with respect to $\theta, x,y$. We have
\begin{align}
\int_\Re f(y)\nabla_\theta p_t(\theta;x,y)\, dy
&=\nabla_\theta\E_x^\theta f(X_t^\theta)=\E_x^\theta f'(X_t^\theta)(\nabla_\theta X_t^\theta)
\nn\\
&=\E_x^\theta \dif_{t,\rho} f(X_t^\theta)\left({\prt_\theta X_t\over \dif_{t,\rho} X_t}\right)
=\E_x^\theta f(X_t^\theta)\Xi^1_{t,\rho}
\nn\\
&=\E_x^\theta f(X_t^\theta)g_t^\theta(x,X_t^\theta)  =\int_\Re f(y)  g_t(\theta;x,y) p_t(\theta;x,y)\, dy.
\nonumber
\end{align}
Hence the formula (\ref{g}) is actually equivalent to the following: for every compactly supported $f\in C^1(\Re)$,
\be\label{weakg}
\nabla_\theta \E f(X_t^\theta)=\E f(X_t^\theta) \Xi^{\theta}_{t,\rho},
\ee
and to prove  (\ref{modif_g}) it is sufficient to prove (\ref{weakg}) with the modified Malliavin weight $ \Xi^{\theta}_{t}$
instead of  $\Xi^{\theta}_{t,\rho}$. To do that, we exploit  an approximation procedure, hence we rewrite (\ref{weakg}) in an integral form, which is convenient for approximation purposes:
\be\label{to} \E
f\left(X^{\theta+v}_t\right)-\E
f\left(X^{\theta}_t\right)=\int_{0}^{1}\E
\left\{f\left(X^{\theta+sv}_t\right)(\Xi^{\theta+sv}_{t,\rho}, v)\right\}ds. \ee

 Because
$$
\prt_\beta X^\theta=t, \quad \prt_\gamma X^\theta=Z_t,
$$
we have
$$
\dif_{t, \rho}(\prt_\beta X^\theta)=0, \quad \dif_{t, \rho}(\prt_\gamma X^\theta)= \dif_{t, \rho}Z_t={1\over \gamma}\dif_{t, \rho}X_t^\theta,$$
therefore
$$\Xi^\theta_{t, \rho}=\left({t \delta_{t, \rho}(1)\over \dif_{t, \rho} X^{\theta}_t}+{t
\dif^2_{t, \rho} X^{\theta}_t \over (\dif_{t, \rho}
X^{\theta}_t)^2},{Z_t \delta_{t, \rho}(1)\over
\dif_{t, \rho} X^{\theta}_t}+{Z_t \dif^2_{t, \rho} X^{\theta}_t \over (\dif_{t, \rho}
X^{\theta}_t)^2}-{1\over \gamma}\right)^\top.$$

Now we proceed with the first approximation step as follows. Fix some $\rho_1\in C^2(\Re), \rho_1(u)\geq 0$ such that
$$
\rho_1(u)=\begin{cases} u^{2},&|u|\leq 1;\\
0,&|u|\geq 2,\end{cases}
$$
and define
$$
\rho_N(u)=N^2\rho_1(u/N).
$$
Observe that, for $t$ fixed and $N$ large enough,  $\rho_N(u)=u^2$ for $|u|\leq t^{1/\alpha}$, hence
$$\dif_{t, \rho_N}X^\theta_t\geq \gamma\kappa_t.$$
Next, there exists a constant $C$ such that
$$
\rho_N(u)\leq Cu^2, \quad |\rho_N'(u)|\leq C|u|,
$$
and therefore
$$
\left|{\chi_{\rho_N}(u)\over u}\right|\leq C(\tau(u)+1).
$$
Then, repeating literally the calculations from Section \ref{s32}, we can obtain a bound similar to (\ref{un_xi_bound}) for $t$ fixed, but a family of weights $\Xi_{t, \rho_N}^\theta,N\geq 1$ is considered instead:
\be\label{12}
\sup_{N\geq 1, \theta\in \tilde \Theta}\E\Big|\Xi_{t, \rho_N}^\theta\Big|^{2+\delta_1}<\infty, \quad \delta_1<\delta\wedge \delta'
\ee
(here $\delta$ comes from H2, and $\delta'$ comes from (\ref{add})). Hence the family $\{\Xi_{t, \rho_N}^\theta, N\geq 1, \theta\in \tilde \Theta\}$ is uniformly integrable. It is straightforward to see that
$$
\Xi_{t, \rho_N}^{\theta_N}\to \Xi_{t}^\theta, \quad N\to \infty
$$
with probability 1 for any sequence $\theta_N\to \theta\in \Theta$. Combined with the above uniform integrability, this shows that
$$
\Xi_{t, \rho_N}^{\theta}\to \Xi_{t}^\theta, \quad N\to \infty
$$
in $L_1(\Omega, \P)$ uniformly with respect to $\theta\in \tilde \Theta$. Hence we can pass to the limit in (\ref{to}) as $N\to \infty$ and get the required identity (\ref{weakg}) with  the modified Malliavin weight $\Xi_{t}^\theta$. This proves the representation (\ref{modif_g}) under the additional moment assumption (\ref{add}).

The second approximation step is aimed to remove the  assumption (\ref{add}), and is similar to the above one. Consider a family of processes $Z^L, L\geq 1$ with L\'evy measures
$$
\mu_L(du)=m_L(u)du, \quad m_L(u)=m(u)e^{-u^2/L};
$$
the $Z^{L}$-driven versions of the processes $X^\theta_t$ and $\Xi^\theta_t$ are also specified by the superscript $L$:
$X^{\theta, L}$, $\Xi^{\theta, L}$.
Because $|u|^{2+\delta}e^{-u^2/L}\leq C$,
every $\mu_L$ satisfies (\ref{add}). In addition, it is an easy calculation to show that conditions \textbf{H1}, \textbf{H2} are satisfied for
$\mu_L$ uniformly with respect to $L\geq 1$. Hence we have the following:
\begin{itemize}
  \item[(a)] for every $L\geq 1$, (\ref{to}) holds true with $X^\theta_t$ and $\Xi^\theta_t$ replaced by $X^{\theta, L}_t$ and $\Xi^{\theta, L}_t$,
  respectively;
  \item[(b)] for every $t\in (0,1]$, the family $\Xi^{\theta, L}_t, \theta\in \tilde\Theta, L\geq 1$ satisfies an analogue of (\ref{12}) uniformly with respect to $L\geq 1$ (to prove this, one should repeat literally the calculations from Section \ref{s32}).
\end{itemize}

Finally, it is straightforward to see that $(X^{\theta, L}_t, \Xi^{\theta, L}_t)$ weakly converge to $(X^{\theta}_t, \Xi^{\theta}_t)$ as $L\to \infty$. Since the family $\{\Xi^{\theta, L}_t\}$ is uniformly integrable by the above property (b), we can pass to the limit in the relation (\ref{to}) for  $X^{\theta, L}_t, \Xi^{\theta, L}_t$, and get finally (\ref{to}) for  $X^{\theta}_t, \Xi^{\theta}_t$. This proves (\ref{modif_g}) and completes the proof of Theorem \ref{thm3}.\qed


\appendix
\section{Proof of Proposition \ref{prop1}}

(1) Because $U$ is negligible (see (\ref{neglig})), we can restrict our considerations to the variables
$$
\zeta_{\alpha,t}=t^{-1/\alpha}(Z_t+c_t).
$$
Their characteristic functions have the form
$\E e^{i\lambda \zeta_{\alpha,t}}=e^{\psi_{\alpha,t}(\lambda)}$, where
$$\ba
\psi_{\alpha,t}(\lambda)&=t\int_{\Re}\left(e^{i\lambda t^{-1/\alpha }u}-1-i\lambda t^{-1/\alpha } u1_{|u|\leq 1}\right)\, \mu(du)+i\lambda c_tt^{-1/\alpha}
\\&=t\int_{\Re}\left(e^{i\lambda t^{-1/\alpha }u}-1-i\lambda t^{-1/\alpha } u1_{|u|\leq t^{1/\alpha}}\right)\, \mu(du);
\ea
$$
in the last identity we have used the formula (\ref{c_t}) for $c_t$. Changing the variable $v=ut^{-1/\alpha}$, we get
$$
\psi_{\alpha,t}(\lambda)=\int_{\Re}\left(e^{i\lambda v}-1-i\lambda v1_{|v|\leq 1}\right)\, \mu_t(dv),
$$
where
$
\mu_t(dv)
$
has the density
$$
m_t(v)=t^{1+1/\alpha} m(t^{1/\alpha} v).
$$
By \textbf{H1}, for every $\eps>0$ there exists $u_\eps>0$ such that
$$
(1-\eps)m_{\alpha, C_\pm}(v)\leq m_t(v)\leq (1+\eps)m_{\alpha, C_\pm}(v), \quad |v|\leq t^{-1/\alpha} u_\eps.
$$
On the other hand, the term $\left(e^{i\lambda v}-1-i\lambda v1_{|v|\leq 1}\right)$ is bounded, and
$$
\mu_t\Big(\{v: |v|>t^{-1/\alpha} u_\eps\}\Big)=t\mu\Big(\{u: |u|> u_\eps\}\Big)\to 0, \quad t\to 0.
$$
Using that, one can easily derive
$$
\psi_{\alpha,t}(\lambda)\to \psi_{\alpha,C_\pm}(\lambda):=\int_{\Re}\left(e^{i\lambda v}-1-i\lambda v1_{|v|\leq 1}\right)m_{\alpha,C_\pm}(v)\, dv,\quad t\to 0.
$$
Because the characteristic function of $Z_1^{\alpha,C_\pm}$ equals $e^{\psi_{\alpha,C_\pm}(\lambda)}$, this completes the proof.

\medskip

(2) Consider first the case
 $U\equiv 0$; now $\phi_{\alpha,t}$ does not depend on $\theta$, and we omit $\theta$ in the notation. We would like to apply the inverse Fourier transform representation for $\phi_{\alpha,t}$ and its derivatives:
$$
\phi_{\alpha,t}(x)={1\over 2\pi}\int_{\Re}e^{-i\lambda x+\psi_{\alpha,t}(\lambda)}\, d\lambda,
$$
$$
(\prt_x)^k\phi_{\alpha,t}(x)={1\over 2\pi}\int_{\Re}(-i\lambda)^ke^{-i\lambda x+\psi_{\alpha,t}(\lambda)}\, d\lambda,
$$
To do that,  we have to verify that the functions under the integrals are absolutely integrable. We have
$$
|e^{-i\lambda x+\psi_{\alpha,t}(\lambda)}|\leq e^{\mathrm{Re}\,\psi_{\alpha,t}(\lambda)},
$$
$$\mathrm{Re}\,\psi_{\alpha,t}(\lambda)=t\int_{\Re}(\cos (t^{-1/\alpha}\lambda u)-1)\mu(du)\leq t\int_{t^{-1/\alpha}|\lambda u|<1}(\cos (t^{-1/\alpha}\lambda u)-1)\mu(du).
$$
Then by \textbf{H1} there exist $c_1, c_2>0$ such that
$$
|e^{-ix\lambda+\psi_{\alpha,t}(\lambda)}|\leq c_1e^{-c_2|\lambda|^\alpha}, \quad x,\lambda\in \Re, \quad t\in (0,1].
$$
This proves  existence of $\phi_{\alpha,t}(x)$ and all  its derivatives. Moreover, we have
$$
\sup_{x\in \Re, t\in (0,1]}|\phi_{\alpha,t}'(x)|<\infty, \quad \sup_{x\in \Re, t\in (0,1]}|\phi_{\alpha,t}''(x)|<\infty.
$$

Now we come back to the case of non-zero $U$. Because the law of $\zeta_{\alpha,t}^\theta$ is a convolution of the laws of $\zeta_{\alpha,t}$ and $\gamma^{-1}U_t$, the above bound can be extended:
\be\label{bound}
\sup_{\theta\in \tilde \Theta}\sup_{x\in \Re, t\in (0,1]}|\phi_{\alpha,t}'(\theta;x)|<\infty, \quad \sup_{\theta\in \tilde \Theta}\sup_{x\in \Re, t\in (0,1]}|\phi_{\alpha,t}''(\theta;x)|<\infty.
\ee
In addition, $\zeta_{\alpha,t}^\theta\Rightarrow Z_1^{\alpha, C_\pm}$ uniformly in $\theta\in \tilde \Theta, U\in \mathfrak{U}$: this follows from the statement (1) and the fact that $\gamma^{-1}t^{-1/\alpha}U_t$ is uniformly negligible. This convergence and the first (resp. second) bound in (\ref{bound}) provide the first (resp. second) convergence in (\ref{conv}).
\qed



\bigskip
\subsection*{Acknowledgement} 
The authors are grateful to the anonymous referees for their valuable comments, which led to substantial improvement of the earlier version.


\end{document}